\documentclass[twoside,11pt]{article}
\usepackage{hyperref,amsthm,amsfonts,amsmath,graphicx,color,marvosym}
\usepackage[lmargin=30mm,rmargin=30mm,tmargin=30mm,bmargin=30mm]{geometry}
\usepackage[sort&compress,numbers]{natbib}

\renewcommand{\baselinestretch}{1.08}
\theoremstyle{plain} 
\newtheorem{theorem}{Theorem}
\newtheorem{proposition}{Proposition}

\newcommand{\bt}[2][]{\ensuremath{\textup{\textsf{bt}}_{#1}(#2)}}
\newcommand{\as}[1]{\ensuremath{\protect{{\widehat{#1}}}}}

\newcommand{\thmlabel}[1]{\label{thm:#1}}
\newcommand{\thmref}[1]{Theorem~\ref{thm:#1}}
\newcommand{\figlabel}[1]{\label{fig:#1}}
\newcommand{\figref}[1]{Figure~\ref{fig:#1}}

\newcommand{\Figure}[4][htb]{
  \begin{figure}[#1]
    \vspace*{1ex}
    \begin{center}#3\end{center} \vspace*{-2ex}
    \caption{\figlabel{#2}#4}
  \end{figure}}

\begin{document}

\title{{\bf On the book thickness of $k$-trees}}

\author{Vida Dujmovi\'c\footnote{School of Computer Science, Carleton
    University, Ottawa, Canada (\texttt{vida@scs.carleton.ca}).}  \and
  David R.\ Wood\footnote{Department of Mathematics and Statistics,
    The University of Melbourne, Melbourne, Australia
    (\texttt{woodd@unimelb.edu.au}). Supported by a QEII Research
    Fellowship from the Australian Research Council.}\\[2ex]
}

\maketitle

\begin{abstract}
  Every $k$-tree has book thickness at most $k+1$, and this bound is
  best possible for all $k\geq3$. Vandenbussche et al. (2009) proved
  that every $k$-tree that has a smooth degree-3 tree decomposition
  with width $k$ has book thickness at most $k$. We prove this result
  is best possible for $k\geq 4$, by constructing a $k$-tree with book
  thickness $k+1$ that has a smooth degree-4 tree decomposition with
  width $k$.  This solves an open problem of Vandenbussche et
  al. (2009)
\end{abstract}

%%%%%%%%%%%%%%%%%%%%%%%%%%%%%%%%%%%%%%%%%%%%%%%%%%%%%%%%%%%%%%
\section{Introduction}
%%%%%%%%%%%%%%%%%%%%%%%%%%%%%%%%%%%%%%%%%%%%%%%%%%%%%%%%%%%%%%

Consider a drawing of a graph\footnote{We consider simple, finite,
  undirected graphs $G$ with vertex set $V(G)$ and edge set $E(G)$. We
  employ standard graph-theoretic terminology; see
  \citep{Diestel00}. For disjoint $A,B\subseteq V(G)$, let $G[A;B]$
  denote the bipartite subgraph of $G$ with vertex set $A\cup B$ and
  edge set $\{vw\in E(G):v\in A, w\in B\}$. } $G$ in which the
vertices are represented by distinct points on a circle in the plane,
and each edge is a chord of the circle between the corresponding
points. Suppose that each edge is assigned one of $k$ colours such
that crossing edges receive distinct colours. This structure is called
a \emph{$k$-page book embedding} of $G$: one can also think of the
vertices as being ordered along the spine of a book, and the edges
that receive the same colour being drawn on a single page of the
book without crossings. The \emph{book thickness} of $G$, denoted by
\bt{G}, is the minimum integer $k$ for which there is a $k$-page book
embedding of $G$. Book embeddings, first defined by \citet{Ollmann73},
are ubiquitous structures with a variety of applications; see
\citep{DujWoo-DMTCS04} for a survey with over 50 references. A book
embedding is also called a \emph{stack layout}, and book thickness is
also called \emph{stacknumber}, \emph{pagenumber} and \emph{fixed
  outerthickness}.

This paper focuses on the book thickness of $k$-trees. A vertex $v$ in
a graph $G$ is \emph{$k$-simplicial} if its neighbourhood, $N_G(v)$, is
a $k$-clique.  For $k\geq1$, a \emph{$k$-tree} is a graph $G$ such
that either $G\simeq K_{k+1}$, or $G$ has a $k$-simplicial vertex $v$
and $G\setminus v$ is a $k$-tree. In the latter case, we say that $G$
is obtained from $G-v$ by \emph{adding $v$ onto} the $k$-clique
$N_G(v)$.

What is the maximum book thickness of a $k$-tree? Observe that 1-trees
are precisely the trees. \citet{BK79} proved that every 1-tree has a
$1$-page book embedding. In fact, a graph has a 1-page book embedding
if and only if it is outerplanar \citep{BK79}. 2-trees are the
edge-maximal series-parallel graphs. \citet{RM-COCOON95}
proved that every series parallel graph, and thus every 2-tree, has a
$2$-page book embedding (also see \citep{GDLW-Algo06}). This bound is
best possible, since $K_{2,3}$ is series parallel and is not
outerplanar. \citet{GH-DAM01} proved that every $k$-tree has a
$(k+1)$-page book embedding; see \citep{DujWoo-DCG07} for an
alternative proof. \citet{GH-DAM01} also conjectured that every
$k$-tree has a $k$-page book embedding. This conjecture was refuted by
\citet{DujWoo-DCG07}, who constructed a $k$-tree with book thickness
$k+1$ for all $k\geq3$. \citet{vandenbussche:1455} independently
proved the same result. Therefore the maximum book thickness of a
$k$-tree is $k$ for $k\leq2$ and is $k+1$ for $k\geq3$.

Which families of $k$-trees have $k$-page book embeddings?
\citet{TY-DM02} proved that every graph with pathwidth $k$ has a
$k$-page book embedding (and there are graphs with pathwidth $k$ and
book thickness $k$). This result is equivalent to saying that every
$k$-tree that has a smooth degree-2 tree decomposition\footnote{See
  \citep{Diestel00} for the definition of tree decomposition and
  treewidth. Note that $k$-trees are the edge maximal graphs with
  treewidth $k$. A tree decomposition of width $k$ is \emph{smooth} if
  every bag has size exactly $k+1$ and any two adjacent bags have
  exactly $k$ vertices in common. Any tree decomposition of a graph
  $G$ can be converted into a smooth tree decomposition of $G$ with
  the same width. A tree decomposition is \emph{degree-$d$} if the
  host tree has maximum degree at most $d$.} of width $k$ has a $k$-page book
embedding. \citet{vandenbussche:1455} extended this result by showing
that every $k$-tree that has a smooth degree-3 tree decomposition of
width $k$ has a $k$-page book embedding. \citet{vandenbussche:1455} then
introduced the following natural definition. Let $m(k)$ be the maximum
integer $d$ such that every $k$-tree that has a smooth degree-$d$
tree decomposition of width $k$ has a $k$-page book
embedding. \citet{vandenbussche:1455} proved that $3\leq m(k)\leq
k+1$, and state that determining $m(k)$ is an open problem. However,
it is easily seen that the $k$-tree with book thickness $k+1$
constructed in \citep{DujWoo-DCG07} has a smooth degree-5
tree decomposition with width $k$. Thus $m(k)\leq 4$ for all
$k\geq3$. The main result of this note is to refine the construction
in \citep{DujWoo-DCG07} to give a $k$-tree with book thickness $k+1$
that has a smooth degree-4 tree decomposition with width $k$ for all
$k\geq4$. This proves that $m(k)=3$ for all $k\geq4$. It is open
whether $m(3)=3$ or $4$. We conjecture that $m(3)=3$. 

%%%%%%%%%%%%%%%%%%%%%%%%%%%%%%%%%%%%%%%%%%%%%%%%%%%%%%%%%%%%%%
\section{Construction}
%%%%%%%%%%%%%%%%%%%%%%%%%%%%%%%%%%%%%%%%%%%%%%%%%%%%%%%%%%%%%%

\begin{theorem}
  \thmlabel{bt-smooth} For all $k\geq 4$ and $n\geq 11(2k^2+1)+k$,
  there is an $n$-vertex $k$-tree $Q$, such that $\bt{Q}=k+1$ and $Q$
  has a smooth degree-4 tree decomposition of width $k$.
\end{theorem}

\begin{proof}
  Start with the complete split graph $K^\star_{k,2k^2+1}$. That is,
  $K^\star_{k,2k^2+1}$ is the $k$-tree obtained by
  adding a set $S$ of $2k^2+1$ vertices onto a $k$-clique
  $K=\{u_1, u_2, \dots, u_k\}$, as illustrated in \figref{Gks}.
  \Figure{Gks}{\includegraphics{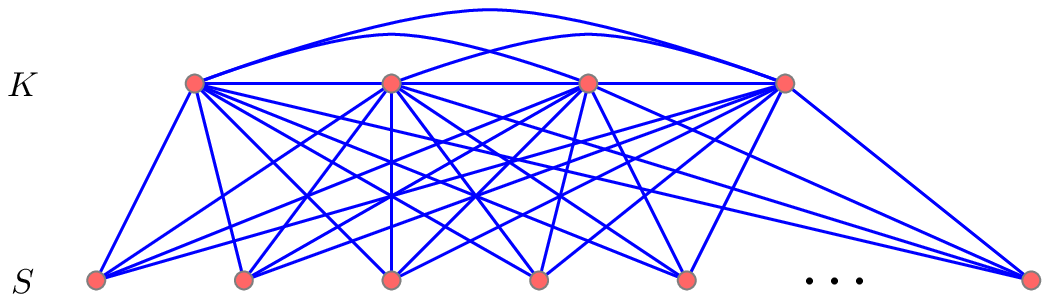}}{The complete split graph
    $K^\star_{4,|S|}$.}
  For each vertex $v\in S$ add a vertex onto the $k$-clique
  $(K\cup\{v\})\setminus\{u_1\}$. Let $T$ be the set of vertices added
  in this step.  For each $w\in T$, if $v$ is the neighbour of $w$ in
  $S$, then add a set $T_2(w)$ of three simplicial vertices onto the
  $k$-clique $(K\cup\{v,w\})\setminus\{u_1,u_2\}$, add a set $T_3(w)$
  of three simplicial vertices onto the $k$-clique
  $(K\cup\{v,w\})\setminus\{u_1,u_3\}$, and add a set $T_3(w)$ of
  three simplicial vertices onto the $k$-clique
  $(K\cup\{v,w\})\setminus\{u_1,u_4\}$.  For each $w\in T$, let
  $T(w):=T_2(w)\cup T_3(w)\cup T_4(w)$.  By construction, $Q$ is a
  $k$-tree, and as illustrated in \figref{DegreeFour}, $Q$ has a
  smooth degree-4 tree decomposition of width $k$.

  \Figure{DegreeFour}{\includegraphics{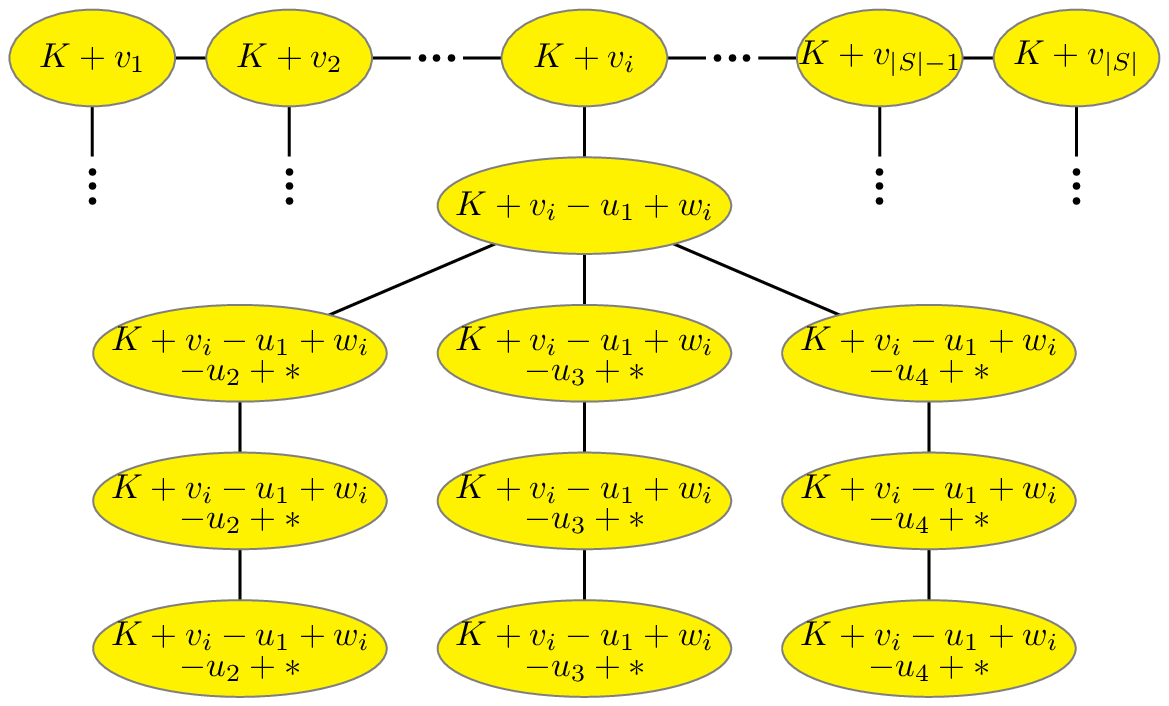}}{A smooth degree-4
    tree decomposition of $Q$.}

  It remains to prove that $\bt{Q}\geq k+1$. Suppose, for the sake of
  contradiction, that $Q$ has a $k$-page book embedding. Say the edge
  colours are $1,2,\dots,k$.  For each ordered pair of vertices
  $v,w\in V(Q)$, let $\as{vw}$ be the list of vertices in clockwise
  order from $v$ to $w$ (not including $v$ and $w$).

  Say $K=(u_1,u_2,\dots,u_k)$ in anticlockwise order. Since there are
  $2k^2+1$ vertices in $S$, by the pigeonhole principle, without loss
  of generality, there are at least $2k+1$ vertices in
  $S\cap\as{u_1u_k}$. Let $(v_1,v_2,\dots,v_{2k+1})$ be $2k+1$
  vertices in $S\cap\as{u_1u_k}$ in clockwise order.

  Observe that the $k$ edges $\{u_iv_{k-i+1}:1\leq i\leq k\}$ are
  pairwise crossing, and thus receive distinct colours, as illustrated
  in \figref{LowerBoundProof}(a). Without loss of generality, each
  $u_iv_{k-i+1}$ is coloured $i$. As illustrated in
  \figref{LowerBoundProof}(b), this implies that $u_1v_{2k+1}$ is
  coloured $1$, since $u_1v_{2k+1}$ crosses all of
  $\{u_iv_{k-i+1}:2\leq i\leq k\}$ which are coloured
  $2,3,\dots,k$. As illustrated in \figref{LowerBoundProof}(c), this
  in turn implies that $u_2v_{2k}$ is coloured $2$, and so on. By an
  easy induction, $u_iv_{2k+2-i}$ is coloured $i$ for each
  $i\in\{1,2,\dots,k\}$, as illustrated in
  \figref{LowerBoundProof}(d). It follows that for all
  $i\in\{1,2,\dots,k\}$ and $j\in\{k-i+1,k-i+2,\dots,2k+2-i\}$, the
  edge $u_iv_j$ is coloured $i$, as illustrated in
  \figref{LowerBoundProof}(e). Moreover, as illustrated in
  \figref{LowerBoundProof}(f):

  \smallskip If $qu_i\in E(Q)$ and $q\in\as{v_kv_{k+2}}$, then $qu_i$
  is coloured $i$.\hfill ($\star$) \smallskip

  \Figure{LowerBoundProof}{\includegraphics{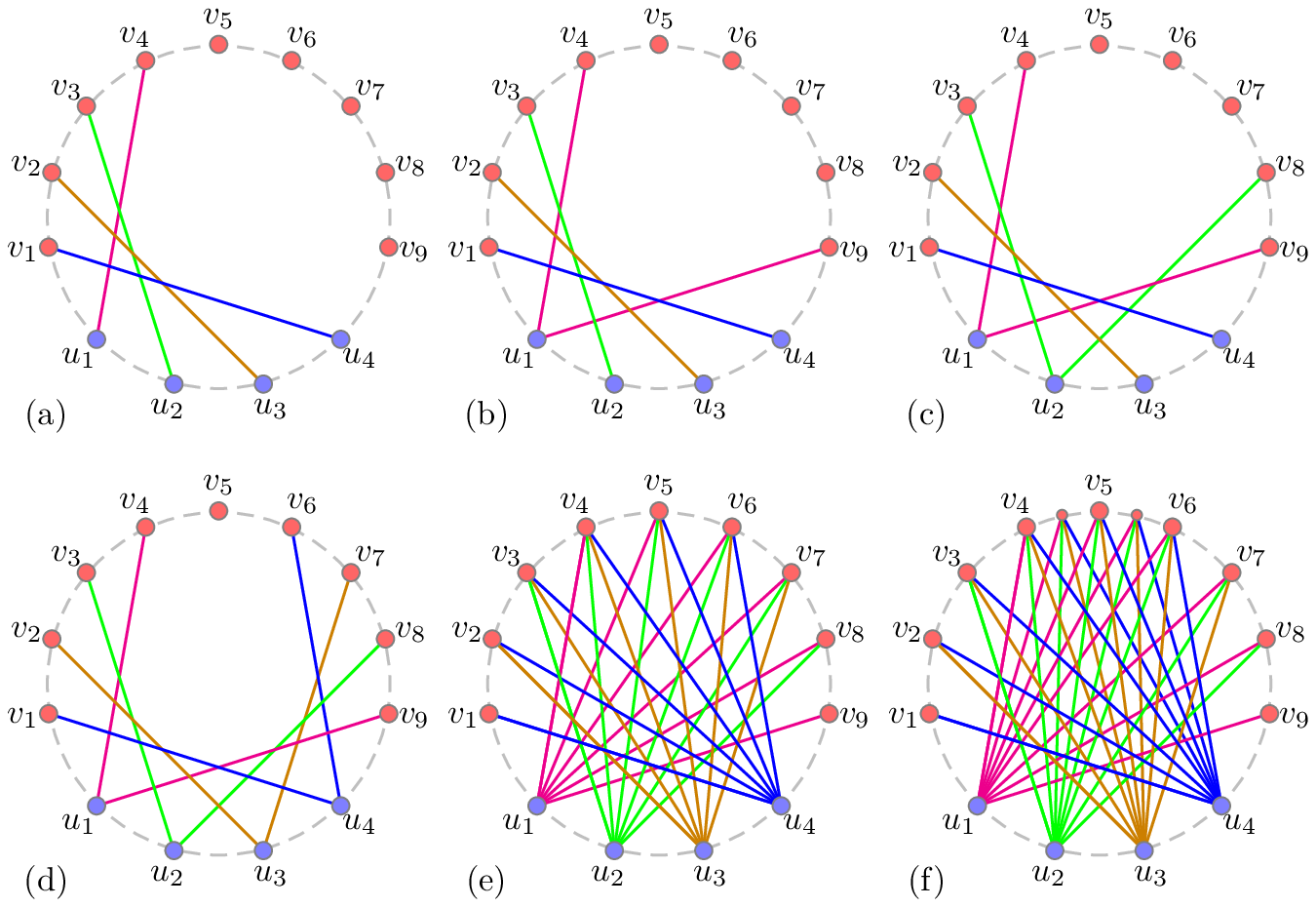}}{Illustration
    of the proof of \thmref{bt-smooth} with $k=4$.}

  Note that the argument up to now is the same as in
  \citep{DujWoo-DCG07}. Let $w$ be the vertex in $T$ adjacent to
  $v_{k+1}$. Vertex $w$ is in $\as{v_{k}v_{k+2}}$, as otherwise the
  edge $wv_{k+1}$ crosses $k$ edges of $Q[\{v_{k},v_{k+2}\};K]$ that
  are all coloured differently. Without loss of generality, $w$ is in
  $\as{v_{k}v_{k+1}}$. Each vertex $x\in T(w)$ is in
  $\as{v_{k}v_{k+1}}$, as otherwise $xw$ crosses $k$ edges in
  $Q[\{v_k,v_{k+1}\};K]$ that are all coloured differently. 
  Therefore, all nine vertices in $T(w)$ are in $\as{v_{k}v_{k+1}}$.
  By the pigeonhole principle, at least one of $\as{v_{k}w}$ or
  $\as{wv_{k+1}}$ contains two vertices from $T_i(w)$ and two vertices
  from $T_j(w)$ for some $i,j\in\{2,3,4\}$ with $i\neq j$.
  % at least four vertices of $T(w)$, two of which are not adjacent to
  % $u_i$ and two of which are not adjacent to $u_j$, where $i\not=j$
  % and $i,j\in\{2,3,4\}$.
  Let $x_1,x_2,x_3,x_4$ be these four vertices in clockwise order in
  $\as{v_{k}w}$ or $\as{wv_{k+1}}$.

  {\em Case 1}. $x_1$, $x_2$, $x_3$ and $x_4$ are in $\as{v_{k}w}$: By
  ($\star$), the edges in $Q[\{w\};K]$ are coloured
  $2,3,\dots,k$. Thus $x_2v_{k+1}$, which crosses all the edges in
  $Q[\{w\};K]$, is coloured $1$.  At least one of vertices in $\{x_2,
  x_3, x_4\}$ is adjacent to $\{K\setminus \{u_1,u_i\}\}$ and at least
  one to $\{K\setminus \{u_1,u_j\}\}$. Thus, by ($\star$), the edges
  in $Q[\{x_2,x_3,x_4\};K]$ are coloured $2,3,\dots,k$. Thus $x_1w$,
  which crosses all the edges of $Q[\{x_2,x_3,x_4\};K]$ is coloured
  $1$. Thus $x_2v_{k+1}$ and $ x_1w$ cross and are both coloured $1$,
  which is the desired contradiction.

  {\em Case 2}. $x_1$, $x_2$, $x_3$ and $x_4$ are in $\as{wv_{k+1}}$:
  As in Case 1, the edges in $Q[\{x_2,x_3,x_4\};K]$ are coloured
  $2,3,\dots,k$. Thus $x_1v_{k+1}$, which crosses all the edges in
  $Q[\{x_2,x_3,x_4\};K]$, is coloured $1$. Since the edges in
  $Q[\{x_1,x_2,x_3\};K]$ are coloured $2,3,\dots,k$, the edge $x_4w$,
  which crosses all the edges of $Q[\{x_1,x_2,x_3\};K]$, is coloured
  $1$. Thus $x_1v_{k+1}$ and $x_4w$ cross and are both coloured $1$,
  which is the desired contradiction.
 
  Finally, observe that $|V(Q)|=|K|+|S|+|T| +\sum_{w\in Q}|T(w)|=
  |K|+11|S|= k+11(2k^2+1)$. Adding more $k$-simplicial vertices to $Q$
  does not reduce its book thickness. Moreover, it is simple to verify
  that the graph obtained from $Q$ by adding simplicial vertices onto
  $K$ has a smooth degree-4 tree decomposition of width $k$. Thus for
  all $n\geq11(2k^2+1)+k$, there is a $k$-tree $G$ with $n$ vertices
  and $\bt{G}=k+1$ that has the desired tree decomposition.
\end{proof}

%%%%%%%%%%%%%%%%%%%%%%%%%%%%%%%%%%%%%%%%%%%%%%%%%%%%%%%%%%%%%%
\section{Final Thoughts}
%%%%%%%%%%%%%%%%%%%%%%%%%%%%%%%%%%%%%%%%%%%%%%%%%%%%%%%%%%%%%%

For $k\geq3$, the minimum book thickness of a $k$-tree is $\lceil\frac{k+1}{2}\rceil$ (since every $k$-tree contains $K_{k+1}$, and $\bt{K_{k+1}}=\lceil\frac{k+1}{2}\rceil$; see \citep{BK79}). However, we now show that the range of book thicknesses of sufficiently large $k$-trees is very limited. 

\begin{proposition}
\label{Range}
Every $k$-tree $G$ with at least $\frac{1}{2}k(k+1)$ vertices has book thickness
$k-1$, $k$ or $k+1$.
\end{proposition}

\begin{proof}
\citet{GH-DAM01} proved that $\bt{G}\leq k+1$. It remains to prove
that $\bt{G}\geq k-1$ assuming $|V(G)|\geq \frac{k(k+1)}{2}$. Numerous
authors \citep{BK79,CD84,Keys75} observed that
$|E(G)|<(\bt{G}+1)|V(G)|$ for every graph $G$. Thus
$$(k-1)|V(G)|\leq k|V(G)|-\tfrac{1}{2}k(k+1)=|E(G)|<(\bt{G}+1)|V(G)|\enspace.$$
Hence $k-1<\bt{G}+1$. Since $k$ and $\bt{G}$ are integers, $\bt{G}\geq
k-1$.
\end{proof} 

We conclude the paper by discussing some natural open problems regarding the computational complexity of calculating the book thickness for various classes of graphs.

Proposition~\ref{Range} begs the question: Is there a characterisation
of the $k$-trees with book thickness $k-1$, $k$ or $k+1$? And somewhat
more generally, is there a polynomial-time algorithm to determine the
book thickness of a given $k$-tree?  Note that the $k$-th power of
paths are an infinite class of $k$-trees with book thickness $k-1$;
see \citep{SGB95}.

%\notice{Does every sufficiently large $k$-tree that is not a power of a path have book thickness at least $k$?}

$k$-trees are the edge-maximal chordal graphs with no $(k+2)$-clique,
and also are the edge-maximal graphs with treewidth $k$. Is there a polynomial-time algorithm to determine the book thickness of a given chordal graph?  Is there a polynomial-time  algorithm to determine the book thickness of a given graph with bounded treewidth?

%%%%%%%%%%%%%%%%%%%%%%%%%%%%%%%%%%%%%%%%%%%%%%%%%%%%%%%%%%%%%%
%\bibliographystyle{myNatbibStyle}
%\bibliography{myBibliography,myConferences}
%%%%%%%%%%%%%%%%%%%%%%%%%%%%%%%%%%%%%%%%%%%%%%%%%%%%%%%%%%%%%%

\def\cprime{$'$} \def\soft#1{\leavevmode\setbox0=\hbox{h}\dimen7=\ht0\advance
  \dimen7 by-1ex\relax\if t#1\relax\rlap{\raise.6\dimen7
  \hbox{\kern.3ex\char'47}}#1\relax\else\if T#1\relax
  \rlap{\raise.5\dimen7\hbox{\kern1.3ex\char'47}}#1\relax \else\if
  d#1\relax\rlap{\raise.5\dimen7\hbox{\kern.9ex \char'47}}#1\relax\else\if
  D#1\relax\rlap{\raise.5\dimen7 \hbox{\kern1.4ex\char'47}}#1\relax\else\if
  l#1\relax \rlap{\raise.5\dimen7\hbox{\kern.4ex\char'47}}#1\relax \else\if
  L#1\relax\rlap{\raise.5\dimen7\hbox{\kern.7ex
  \char'47}}#1\relax\else\message{accent \string\soft \space #1 not
  defined!}#1\relax\fi\fi\fi\fi\fi\fi} \def\Dbar{\leavevmode\lower.6ex\hbox to
  0pt{\hskip-.23ex\accent"16\hss}D}

\end{document}